\documentclass[11pt]{article}  
\usepackage{amsmath,amssymb,euscript,graphicx,amsfonts}
 
\newtheorem{theorem}{Theorem}[section]
\newtheorem{proposition}[theorem]{Proposition}
\newtheorem{remark}[theorem]{Remark}
\newtheorem{lemma}[theorem]{Lemma}
\newtheorem{corollary}[theorem]{Corollary}
\newtheorem{definition}[theorem]{Definition}

\def\ww{\widetilde}
 
\def\a{\alpha}
\def\b{\beta}

\def\l{\lambda}

\font\tenbb=msbm10
\font\sevenbb=msbm8
\font\fivebb=msbm5
\newfam\bbfam
\textfont\bbfam=\tenbb \scriptfont\bbfam=\sevenbb
\scriptscriptfont\bbfam=\fivebb
\def\bb{\fam\bbfam}

\def\FF{{\bb F}}

\def\Z{{\bb Z}}

\begin{document}
 
\begin{center}
{\bf \Large  On the mod - $p$ cohomology of $Out(F_{2(p-1)})$}
\end{center}
 
  
\bigskip
\bigskip\noindent
\begin{center}
\begin{tabular}{cc}
{\sc Henry Glover}& {\sc Hans-Werner Henn}\\
{\small Ohio State University} & {\small Universit\'e Louis Pasteur} \\
{\tt glover@math.ohio-state.edu} & {\tt henn@math.u-strasbourg.fr }\\
\end{tabular}

\bigskip
\bigskip
October 28, 2008  
\end{center}
 
\bigskip
 
\abstract
We study the mod-$p$ cohomology of the group $Out(F_n)$ 
of outer automorphisms of the free group $F_n$ 
in the case $n=2(p-1)$ which is the smallest $n$ for which the 
$p$-rank of this group is $2$. For $p=3$ we give a complete computation, 
at least above the virtual cohomological dimension of $Out(F_4)$ (which is 5). 
More precisley, we calculate the equivariant cohomology of the 
$p$-singular part of outer space for $p=3$.   
For a general prime $p> 3$ we give a recursive description in terms of  
the mod-$p$ cohomology of $Aut(F_k)$ for $k\leq p-1$. 
In this case we use the $Out(F_{2(p-1)})$-equivariant cohomology 
of the poset of elementary abelian $p$-subgroups of $Out(F_n)$.  
\endabstract
 
\bigskip
 
\section{Introduction. Background and results.}
\label{sec:intro}
\indent
 
\bigskip
 
Let $F_n$ denote the free group on $n$ generators and let $Out(F_n)$ 
denote its group of outer automorphisms. We are interested in the 
cohomology ring $H^*(Out(F_n);\FF_p)$ with coefficients in the prime 
field $\FF_p$. The case $n=2$ is well understood because of the
isomorphism $Out(F_2)\cong GL_2(\Z)$ and the stable  cohomology of
$Out(F_n)$ has been shown to agree with that of symmetric groups \cite{G}.
Apart from these results the only other complete calculation is that of the 
integral cohomology ring of  $Out(F_3)$ by T. Brady \cite{Bra}.
There has been a fair amount of work on the Farell cohomology of 
$Out(F_n)$, but always in cases where the $p$ - rank
of $Out(F_n)$ is one \cite{GMV,GM,Ch}. (We recall that the $p$ - rank 
of a group $G$ is the maximal $k$ for which $(\Z/p)^k$ embeds into 
$G$.) The $p$ - rank of $Out(F_n)$ is known to be $[\frac{n}{p-1}]$. 
In this paper we consider the case $n=2(p-1)$ for $p$ odd, i.e. the 
first case of $p$ - rank two and compute the mod - $p$ cohomology 
ring, at least above dimension $2n-3$, the virtual cohomological dimension 
of $Out(F_n)$. For $p=3$, $n=4$ our result is completely explicit  and can be 
neatly described in terms of the cohomology of certain finite subgroups 
of $Out(F_4)$. 
In the general case it has a recursive nature, i.e. we express the 
result in terms of automorphism groups of free groups of lower rank.

We will now describe our results. We start by exhibiting certain 
finite subgroups of $Out(F_n)$. For this we identify $F_n$ with the 
fundamental group $\pi_1(R_n)$ of the wedge of $n$ circles which we 
consider as a graph with one vertex and $n$ edges. As such it is also 
called the rose with $n$ leaves and is denoted $R_n$. If $\Gamma$ is 
a finite graph and $\alpha:R_n\to \Gamma$ is an (unpointed) homotopy 
equivalence then $\alpha$ induces an isomorphism 
$\alpha_{*}:Out(F_n)=Out(\pi_1(R_n))\cong Out(\pi_1(\Gamma))$.
The group $Aut(\Gamma)$ of graph automorphisms of $\Gamma$ is a finite group
which for $n>1$ embeds naturally into $Out(\pi_1(\Gamma))$  \cite{SV2}. 
Therefore $\alpha_{*}^{-1}(Aut(\Gamma))$ is a finite subgroup of $Out(F_n)$. 
Note that the choice of (the homotopy class) of $\alpha$ (which is also 
called a marking of $\Gamma$) is important for getting an actual subgroup, 
and that running through the different choices for $\alpha$ amounts to 
running through the different representatives in the same conjugacy 
class of subgroups.

\bigbreak 
Now let $p=3$, $n=4$ and let $R_4$ be the rose with four leaves 
(cf. Figure~\ref{Fig1-1}). Its 
automorphism group can be identified with the wreath product 
$\Z/2\wr \Sigma_4$. Choosing $\alpha=id:R_4\to R_4$ gives us a subgroup
of $Out(F_n)$ which we denote by $G_R$.
 
Next let $\Theta_2$ be the connected graph with two vertices and $3$ 
edges between these two vertices, and let $\Theta_2^{1,1}$ be the 
wedge of $\Theta_2$ with a rose with one leaf attached to each of the 
two vertices of $\Theta_2$ (cf. Figure~\ref{Fig1-1}). 
Its automorphism group is $\Sigma_3\times D_8$ where $D_8$ is the dihedral 
group of order $8$. Choosing any homotopy equivalence between $R$ and
$\Theta_2^{1,1}$ gives us a subgroup of $Out(F_4)$ which we denote by $G_{1,1}$.
 
The automorphism group of the wedge $\Theta_{2}\vee \Theta_{2}$ 
(cf. Figure~\ref{Fig1-1}) is $\Sigma_3\wr \Z/2$. 
After choosing a homotopy equivalence $R_4\to \Theta_2\vee \Theta_2$ 
we get a subgroup of $Out(F_4)$ which we denote by $G_2$.
  
Let $K_{3,3}$ be the Kuratowski graph with two blocks of $3$ vertices 
and $9$ edges which join all vertices from the first block to all 
vertices of the second block (cf. Figure~\ref{Fig1-1}). 
Its automorphism group is again $\Sigma_3\wr \Z/2$ and after choosing
a homotopy equivalence $R_4\to K_{3,3}$ we obtain a subgroup 
$G_{K}$ of $Out(F_4)$.
 
\bigskip

\begin{figure}[h!]
\begin{footnotesize}
\begin{center}
\includegraphics[width=0.60\hsize]{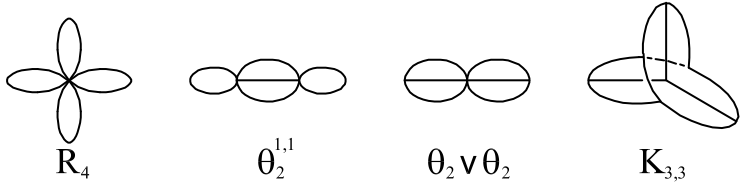}
\caption{\label{Fig1-1}\footnotesize 
The graphs whose automorphism groups
determine $H^*(Out(F_4);\FF_3)$ for $*>5$.}
\end{center}
\end{footnotesize}
\end{figure}

\smallskip

We will see below (cf. section \ref{cell structure}) 
that the subgroup $\Sigma_3\times \Z/2$ of $G_K$ 
which is given by permuting the three vertices in the first block and 
independently two of the three vertices in the second block is 
conjugate in $Out(F_4)$ to the ``diagonal'' subgroup 
$\Delta\Sigma_3\times \Z/2$ of $G_2\cong \Sigma_3\wr \Z/2$. By 
choosing appropriate markings of $\Theta_2\vee \Theta_2$ and 
$K_{3,3}$ we can assume that this subgroup is the same. We denote it by $H$.
Let $G_2*_HG_K$ be the amalgamated product.

\begin{theorem}\label{the:res}
The inclusions of the finite subgroups $G_R$, $G_{1,1}$, 
$G_2$ and $G_K$ into $Out(F_4)$ induces a homomorphism of groups 
$$
G_R*G_{1,1}*(G_2*_HG_K)\to Out(F_4)
$$ 
and the induced map in mod $3$-cohomology  
$$
H^*(Out(F_4);\FF_3)\to H^*(G_R;\FF_3)\times H^*(G_{1,1};\FF_3)
\times H^*(G_2*_HG_K;\FF_3) 
$$
is an isomorphism above dimension $5$.
\end{theorem}
 
The mod-$3$ cohomology of the first two factors is the same as that 
of $\Sigma_3$, i.e. it is isomorphic to the tensor product
$\FF_3[a_4]\otimes \Lambda(b_3)$ of a polynomial algebra generated by 
a class $a_4$ of dimension $4$ and an exterior algebra generated by a 
class $b_3$ of dimension $3$. The cohomology of the amalgamated product  
can also be easily computed and we obtain the following explicit result.

\begin{corollary}
In degrees bigger than $5$ we have
$$
H^*(Out(F_4);\FF_3)\cong 
\prod_{i=1}^2\FF_3[a_4^i]\otimes\Lambda(b_3^i)\times 
\FF_3[r_4,r_8]\otimes\Lambda(s_3)\{1,t_7,\widetilde{t_7},t_8\}\ . 
$$
Here the lower indices indicate the dimensions of the cohomology classes 
and the last factor is described as a free module of rank $4$ over
$\FF_3[r_4,r_8]\otimes\Lambda(s_3)$ on the indicated classes. 
(The full multiplicative structure is given in Proposition 3.3.) 
\end{corollary}
 
We derive these results by analyzing the Borel construction 
$EOut(F_n)\times_{Out(F_n)}K_n$ with respect to the action of 
$Out(F_n)$ on the ``spine $K_n$ of outer space'' \cite{CV}. 
We recall that $K_n$ is a 
contractible simplicial complex of dimension $2n-3$ on which 
$Out(F_n)$ acts simplicially with finite isotropy groups, in 
particular the Borel construction is a classifying space for 
$Out(F_n)$. However, $K_4$ is already very difficult to analyze. We 
replace it therefore by its more accessible $3$-singular locus 
$(K_4)_s$ so that our results really describe the mod-$3$ 
cohomology of $EOut(F_4)\times_{Out(F_4)}(K_4)_s$ which agrees 
with that of $Out(F_4)$ above degree $5$.
 
\bigbreak 
However, for $p>5$ and $n=2(p-1)$, even the $p$ - singular locus of 
$K_n$ becomes too difficult to analyze directly. An alternative 
approach towards the mod-$p$ cohomology of $Out(F_n)$
uses the normalizer spectral sequence \cite{Bro} 
which is associated to the action of
$Out(F_n)$ on the poset $\cal{A}$ of its elementary abelian $p$ - subgroups 
and which also calculates the mod-$p$ cohomology of $Out(F_n)$ 
above its finite virtual cohomological dimension $2n-3$. 
This method has been used by Jensen \cite{J} to study the $p$-primary 
cohomology of $Aut(F_{2(p-1)})$ and it works equally for 
$Out(F_{2(p-1)})$, in fact it is even somewhat simpler.
 
In order to describe our result for $Out(F_{2(p-1)})$ we need to 
describe certain elementary abelian $p$ - subgroups of 
$Out(F_{2(p-1)})$.
For this we consider the following graphs. Let $R_{2(p-1)}$ denote the 
rose with $2(p-1)$ leaves, $\Theta_{p-1}$ denote the graph with two 
vertices with $p$ edges between them, $\Theta_{p-1}^{s,t}$ with 
$s+t=p-1$ denote the graph which is obtained from $\Theta_{p-1}$ by 
attaching a rose with $s$ leaves at one vertex and one with $t$ 
leaves at the other vertex of $\Theta_{p-1}$. Furthermore let 
$\Theta_{p-1}\vee\Theta_{p-1}$ denote the wedge of two $\Theta_{p-1}$ at 
a common vertex, see Figure~\ref{Fig1-2} below.
\vfil\eject

\begin{figure}[ht!]
\begin{footnotesize}
\begin{center}
\includegraphics[width=0.60\hsize]{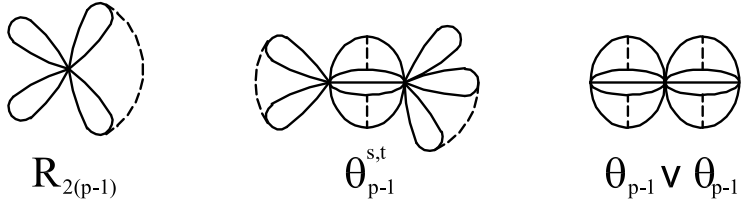}
\caption{\label{Fig1-2}\footnotesize  The graphs which 
determine  $H^*(Out(F_{2p-2});\FF_p)$ for $p>3$ and $*>2p-3$.}
\end{center}
\end{footnotesize}
\end{figure}

After choosing appropriate markings these graphs determine as before 
subgroups
\begin{eqnarray}
G_{R}&\cong & \ \Z/2\wr \Sigma_{2(p-1)} \nonumber \\
G_{s,p-1-s}&\cong &
\begin{cases}  (\Z/2\wr \Sigma_s)\times \Sigma_p\times 
(\Z/2\wr \Sigma_{p-1-s})\ \ &\ s\neq \frac{p-1}{2} \\
\big((\Z/2\wr \Sigma_s)\times \Sigma_p\times 
(\Z/2\wr \Sigma_s)\big)\rtimes\Z/2 \ \ &\ s=\frac{p-1}{2} \\
\end{cases}
\nonumber\\
G_2 &\cong & \ \Sigma_p\wr \Z/2 \ .\nonumber
\end{eqnarray}
(In the third line the action of $\Z/2$ is trivial on the middle 
factor $\Sigma_p$ while it interchanges the other two factors.)
The $p$ - Sylow subgroups in all these groups are elementary abelian. 
We choose representatives and denote them by $E_{R}\cong \Z/p$ resp.
$E_{s,p-1-s}\cong \Z/p$ 
resp. $E_2\cong \Z/p\times \Z/p$. With a 
suitable choice of representatives we can assume that
$E_{0,p-1}$ is one of the two factors of $E_2\cong\Z/p\times\Z/p$.
The structure of the normalizers of these elementary abelian subgroups 
and their relevant intersections is summarized in the 
following result. 

In this result we will abbreviate the normalizer 
$N_{Out(F_{2(p-1)})}(E_R)$ of $E_R$ in $Out(F_{2(p-1)})$ by $N_R$, 
and likewise $N_{Out(F_{2(p-1)})}(E_{s,p-s})$ by 
$N_{s,p-s}$ and $N_{Out(F_{2(p-1)})}(E_2)$ by $N_2$; the normalizer 
of the diagonal subgroup $\Delta(E_2)$ of $E_2$ is abbreviated by $N_{\Delta}$. 
Furthermore $N_{\Sigma_p}(\Z/p)$ 
denotes the normalizer of $\Z/p$ in $\Sigma_p$ and 
$Aut(F_n)$ denotes the group of automorphisms of $F_n$. 
 
\begin{proposition}\label{pro:3}  
 
a) The groups $E_R$, $E_{s,p-1-s}$ for $0\leq s\leq \frac{p-1}{2}$, $E_2$ 
and the diagonal $\Delta(E_2)$ of $E_2$ are pairwise non-conjugate, and 
any elementary abelian $p$-subgroup of $Out(F_{2(p-1)})$ 
is conjugate to one of them. 
\smallbreak 

b) $E_{0,p-1}$ is subconjugate to $E_2$ and neither $E_R$ nor any  
of the $E_{s,p-1-s}$ with $1\leq s\leq \frac{p-1}{2}$ is subconjugate 
to $E_2$. 
\smallbreak 
c) There are canonical isomorphisms   
\begin{eqnarray}
N_R&\cong & \
N_{\Sigma_p}(\Z/p)\times ((F_{p-2}\rtimes Aut(F_{p-2}))\rtimes \Z/2) 
\nonumber \\
N_ {s,p-1-s}&\cong & \ N_{\Sigma_p}(\Z/p)\times
Aut(F_s)\times Aut(F_{p-1-s}) \ \ \ \ \ \text{if}\ \ 0\leq s< \frac{p-1}{2} 
\nonumber\\
N_{s,s}&\cong & \ N_{\Sigma_p}(\Z/p)\times
(Aut(F_s)\wr \Z/2) \ \ \ \ \ \ \ \ \ \ \ \ \ \ \ 
\text{if}\ \ s=\frac{p-1}{2}
\nonumber \\
N_{\Delta}&\cong &\big((\Z/p\times\Z/p)\rtimes(\Z/(p-1)\times\Z/2)\big)
*_{N_{\Sigma_p}(\Z/p))\times \Z/2}{(N_{\Sigma_p}(\Z/p)\times \Sigma_3)}\nonumber \\
N_2&\cong  & \ N_{\Sigma_p}(\Z/p)\wr \Z/2 \ .\nonumber 
\end{eqnarray}
\smallbreak  
d) There are canonical isomorphisms 
\begin{eqnarray}
N_{0,p-1}\cap N_2 &\cong &  
N_{\Sigma_p}(\Z/p)\times N_{\Sigma_p}(\Z/p) \ .\nonumber \\
N_{\Delta}\cap N_2 &\cong &  
(\Z/p\times \Z/p)\rtimes (Aut(\Z/p)\times \Z/2) \ \nonumber 
\end{eqnarray}
where in the semidirect product $Aut(\Z/p)$ acts diagonally on 
$\Z/p\times \Z/p$ and $\Z/2$ acts by interchanging the two factors. 
\end{proposition}

\bigbreak

The evaluation of the normalizer spectral sequence  yields the 
following result.
 
\begin{theorem}\label{the:4}
Let $p>3$ be a prime and $n=2(p-1)$.
 
\smallskip
 
\noindent a) The inclusions of the subgroups $N_R$, $N_{s,p-s}$ and $N_2$ 
into $Out(F_n)$ induces a homomorphism of groups 
$$
N_R*N_{1,p-2}*N_{2,p-3}\ldots *N_{\frac{p-1}{2},\frac{p-1}{2}}*
(N_{0,p-1}*_{N_{0,p-1}\cap N_2}N_2)\to Out(F_n)
$$ 
and the induced map  
$$
H^*(Out(F_n);\FF_p)\to H^*(N_R;\FF_p)
\times \prod_{s=1}^{\frac{p-1}{2}}H^*(N_{s,p-1-s});\FF_p)
\times  H^*(N_{0,p-1}*_{N_{0,p-1}\cap N_2}N_2;\FF_p)
$$
is an isomorphism above dimension $2n-3$.
\smallskip
 
\noindent b) There is an epimorphism of $\FF_p$-algebras 
$$
H^*(N_{0,p-1}*_{N_{0,p-1}\cap N_2}N_2;\FF_p)\to 
H^*(N_2;\FF_p)
$$
whose kernel is isomorphic to the ideal
$H^*(N_{\Sigma_p}(\Z/p);\FF_p)\otimes K_{p-1}$ where
$K_{p-1}$ is the kernel of the restriction map
$H^*(Aut(F_{p-1});\FF_p)\to H^*(\Sigma_p;\FF_p)$. 
\end{theorem}
 
In view of Proposition \ref{pro:3} this result reduces the explicit 
calculation of the mod-$p$ cohomology of $Out(F_{2(p-1)})$ above the finite 
virtual cohomological dimension (v.c.d. in the sequel) to the 
calculation of the cohomology of $Aut(F_s)$ with $s\leq p-1$, i.e. to 
calculations in which the $p$-rank is one. 
If $s<p-1$ these cohomologies are finite, and if $s=p-1$ the cohomology has 
been calculated above the v.c.d. in \cite{GMV}. 
\bigbreak

We remark that the result for $p=3$ can also be derived by the 
normalizer method but requires special considerations which are 
caused by the existence of the graph $K_{3,3}$ and the corresponding 
additional elementary abelian $3$-subgroup $E_K\cong \Z/3\times \Z/3$. 
In addition in this approach a partial analysis of outer space 
$K_4$ still seems needed in order to determine the structure of some 
of the normalizers. We find it therefore preferable to present the 
case $p=3$ via the isotropy spectral sequence for the Borel construction 
of $(K_4)_s$.   
\bigbreak  
 
The remainder of this paper is organized as follows. In Section 2 we 
recall the definition of  the spine $K_n$ of outer space associated with 
$Out(F_n)$ and analyze the $3$-singular part of $K_4$, $(K_4)_s$ 
(cf. \cite{GM}).
In Section 3 we evaluate the isotropy spectral sequence of 
$(K_4)_s$, thus proving Theorem~1.1 and Corollary~1.2.
In Section 4 we discuss the poset of elementary abelian 
$p$-subgroups of $Out(F_{2(p-1)})$ for $p>3$ and prove part (a) and (b) of 
Proposition \ref{pro:3}. In Section 5 we study the normalizers of these elementary 
abelian $p$-subgroups and prove the remaining parts of the same proposition. 
Finally in Section~6 we derive the cohomological consequences and prove 
Theorem~1.4. 
\bigbreak 
 
We wish to thank the University of Heidelberg, the Max Planck 
Institute at Bonn, Ohio State University
and Universit\'e Louis Pasteur at Strasbourg for their support 
while this work was done. Also we wish to thank Dan Gries, 
Karen Li and Klavdija Kutnar for their help in the preparation 
of the graphics in this manuscript.

\section{The spine of outer space and $3$-singular graphs in the rank 
$4$ case}

\subsection{The spine of outer space }

We recall that the spine $K_n$ of outer space is defined as the
geometric realization of the poset of equivalence classes of
marked admissible finite graphs or rank $n$ (i.e. they have the 
homotopy type of the rose $R_n$), where the poset relation is 
generated by collapsing trees \cite{CV}. 
\bigbreak 

In more detail, for our purposes a finite graph $\Gamma$ is a quadruple 
$(V(\Gamma), E(\Gamma), \sigma, t)$ where $V(\Gamma)$ and $E(\Gamma)$  
are finite sets, $\sigma$ is a fixed point free involution of $E(\Gamma)$ and 
$t$ is a map from $E(\Gamma)$ to $V(\Gamma)$. To such a graph one can 
associate canonically a $1$-dimensional $CW$-complex with $0$-skeleton 
$V(\Gamma)$ and with $1$-cells in bijection with the $\sigma$-orbits 
of $E(\Gamma)$. The attaching map of a $1$-cell $e$ is given by $t(e)$, 
the terminal vertex of $e$, and by $t(\sigma(e))$, the initial vertex of $e$.  
A graph $\Gamma$ is called admissible if it 
(i.e. the associated $CW$-complex which will also be denoted by $\Gamma$) 
is connected, all vertices have valency at least $3$, 
and it does not contain any separating edges. A marking of a graph $\Gamma$ 
is a choice of a homotopy equivalence $\a:R_n\to \Gamma$. Two markings 
$R_n\buildrel{\a_i}\over \longrightarrow \Gamma_i$, $i=1,2$  
are equivalent if there exists a homeomomorphism $\varphi:\Gamma_1\to \Gamma_2$ 
such that $\a_2$ and $\varphi \a_1$ are freely homotopic.  

The poset relation is defined as follows:  
$R_n\buildrel{\a_2}\over \longrightarrow \Gamma_2$ is bigger than 
$R_n\buildrel{\a_1}\over \longrightarrow \Gamma_1$ if there exists 
a forest in $\Gamma_2$ such that $\Gamma_1$ is obtained from $\Gamma_2$ 
by collapsing each tree in this forest to a point, 
and $\a_1$ is freely homotopic to 
the composite of the $\a_2$ followed by the collapse map. It is clear 
that this induces a poset structure on equivalence classes of marked graphs. 

The group $Out(F_n)$ can be identified with the
group of free homotopy equivalences of $R_n$ to itself,
and with this identification we obtain a right action of $Out(F_n)$ on
$K_n$, given by precomposing the marking with an unbased 
homotopy equivalence of $R_n$ to itself.
\bigbreak 
 
The isotropy group of (the equivalence class of) a marked graph
$(\Gamma, \a)$ (with marking $\a$) can be identified via 
$\alpha_{*}^{-1}$ with the automorphism group $Aut(\Gamma)$ of 
the underlying graph $\Gamma$. (Note that graph automorphisms 
have to be taken in the sense of the definition of a graph given above, in 
particular they are allowed to reverse edges!) 

If $(\Gamma_i,\alpha_i)$, $0\leq i<k$, is obtained from $(\Gamma_k,\alpha_k)$ 
by collapsing a sequence of forests $\tau_i$ with $\tau_i\supset \tau_{i+1}$ 
then the isotropy group of the $k$ - simplex with vertices 
$(\Gamma_0,\alpha_0), 
\ldots,(\Gamma_k,\alpha_k)$ can be identified with the subgroup of  
$Aut(\Gamma_k)$ which leaves each of the forests $\tau_i$ invariant, 
cf. \cite{SV2}.
\bigbreak 

It is easy to check that an admissible graph of rank $n$ has at most 
$3n-3$ and at least $n$ edges, and the complex $K_n$ has 
dimension $2n-3$.
 
\subsection{The $3$ - singular graphs of rank $4$}

With these preparations we can now start discussing 
the $3$-singular locus $(K_4)_s$ of $K_4$, 
i.e. the subspace of $K_4$ of points whose isotropy 
groups contain non-trivial elements of order $3$. For this we first need to 
determine the graphs with a non-trivial element of order $3$ in its 
automorphism group (we call them $3$ - singular graphs) and then the 
invariant chains of forests in them which are preserved by a 
non-trivial element of order $3$. The necessary analysis is analogous 
to that of the $p$ - singular locus of $K_{p+1}$ for $p>3$ \cite{GM}. 
Our notation partially follows that of \cite{GM}, however we have chosen 
different notation for the graphs  
$\Theta_r^{s,t}$, $K_{3,3}$ (which was labelled $S_1$ in \cite{GM}) and 
$\Theta_2:\Theta_1$ (which was labelled $\Theta_2\ast \Theta_2$ in \cite{GM}). 

\bigbreak 

In this section we give a brief outline how one arrives at the list of 
$3$-singular graphs decribed below in Table~\ref{Tab2-1}.  
We already know that these graphs will have at least $4$ and at most $9$ edges.
\bigbreak

\begin{table}[ht!]
\begin{footnotesize}
\begin{center}
\includegraphics[width=0.48\hsize]{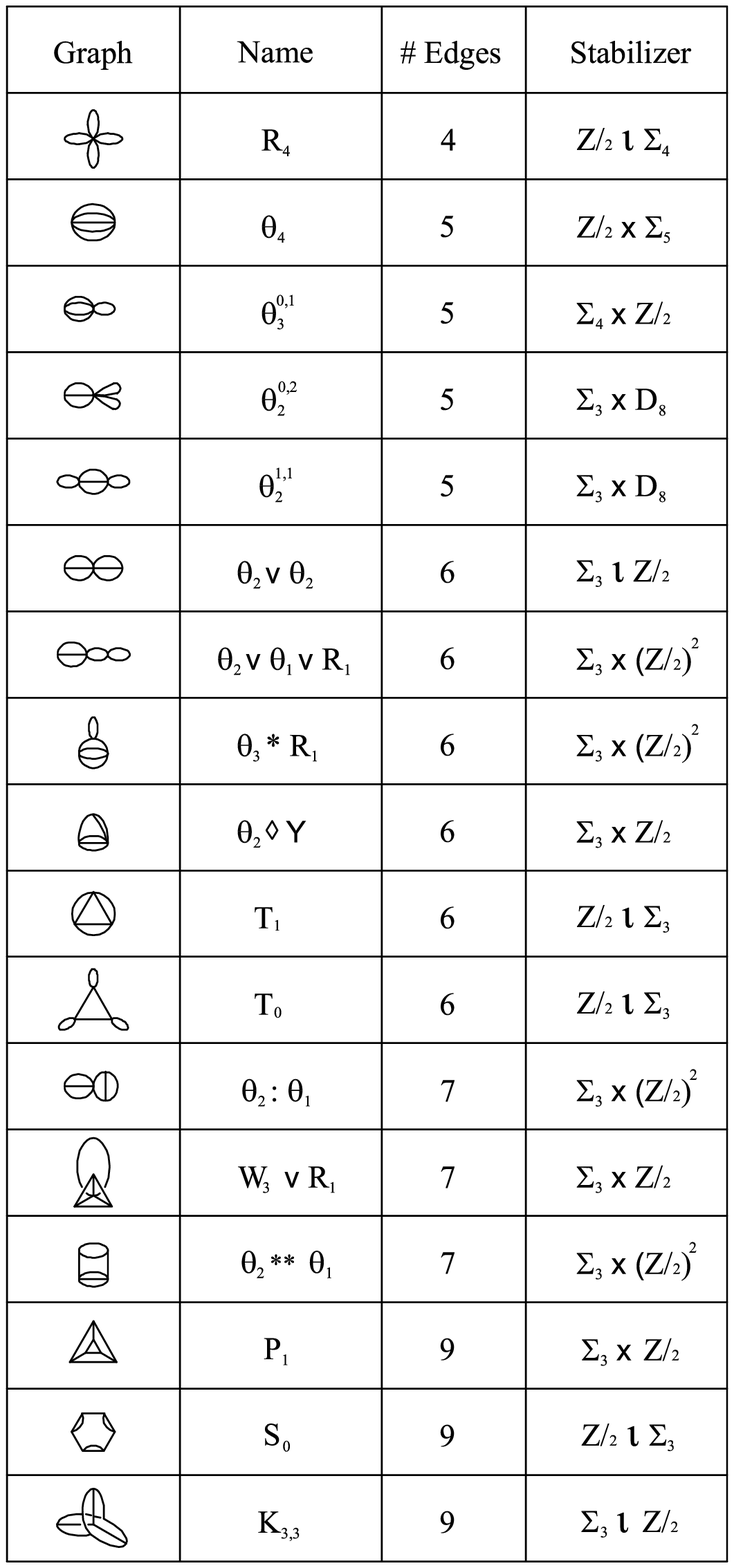}
\caption{\label{Tab2-1}\footnotesize  $3$-singular admissible graphs of 
rank $4$}
\end{center}
\end{footnotesize}
\end{table}

\medskip
 
\noindent {\em 2.2.1\ \ }With $4$ edges we only find the rose $R_4$ 
which is clearly $3$ - singular.  Its automorphism group is clearly 
the wreath product $\Z/2\wr\Sigma_4$.  
 
\medskip
  
\noindent  {\em 2.2.2 \ \ }With $5$ edges we need $2$ vertices 
which are necessarily fixed with respect to any graph automorphism of 
order $3$. In order to make the graph $3$-singular and 
admissible, we need to have at least $3$ edges connecting the two 
vertices. The resulting graphs and corresponding isomorphism groups are

\begin{eqnarray} 
\Theta_4       &\ \ & \Z/2\times \Sigma_5 \nonumber \\
\Theta_3^{0,1} &\ \ & \Sigma_4\times \Z/2 \nonumber\\
\Theta_2^{1,1} &\ \ & \Sigma_3\times D_8 \nonumber\\  
\Theta_2^{0,2} &\ \ & \Sigma_3\times D_8 \ . \nonumber
\end{eqnarray}

\medskip
 
\noindent {\em 2.2.3\ \ }With $6$ edges we need $3$ vertices. 
First we consider the case that a $3$ - Sylow subgroup 
$P$ of $Aut(\Gamma)$ fixes these vertices. If the $P$ - orbit of each
edge is non-trivial then we must have $2$ orbits of length $3$. 
The resulting graph and its automorphism group are 
$$
\Theta_2\vee \Theta_2 \ \ \ \ \ \Sigma_3\wr \Z/2 \ .
$$

If there is an edge which is fixed by $P$ then there there are three 
fixed edges and there is one orbit of edges of length $3$. In this 
case we have one of the following cases with corresponding automorphism groups  

\begin{eqnarray}
\Theta_2\vee\Theta_1\vee R_1  &\ \ & \Sigma_3\times (\Z/2)^2 \nonumber\\
\Theta_3\ast R_1              &\ \ & \Sigma_3\times (\Z/2)^2\nonumber\\
\Theta_2\diamond Y      &\ \ & \Sigma_3\times \Z/2  \ .   \nonumber
\end{eqnarray}
  
Now assume that $P$ permutes all three vertices. Then the number of 
edges between any two vertices is either $1$ or $2$ and we have one 
of the following cases

\begin{eqnarray}
T_0          &\ \ &  \Z/2\wr \Sigma_3 \nonumber\\
T_1          &\ \ &  \Z/2\wr \Sigma_3 \ . \nonumber
\end{eqnarray}

\medskip
\noindent {\em 2.2.4\ \ } Next we consider the case of $7$ edges and $4$ 
vertices. First we assume again that a $3$-Sylow subgroup $P$ of 
$Aut(\Gamma)$ fixes all vertices. Then we can have only one 
non-trivial $P$ orbit of edges (otherwise the graph would not be 
admissible) and we are in one of the following cases
\begin{eqnarray}
\Theta_2 :\Theta_1          &\ \ &  \Sigma_3\times (\Z/2)^2\nonumber\\
\Theta_2{\ast}{\ast}{\Theta_1}  &\ \ & \Sigma_3\times (\Z/2)^2  \ .\nonumber
\end{eqnarray}

If $P$ acts nontrivially on the vertices then we have an orbit of 
length $3$ and one orbit of length $1$. Admissibility forces that there are 
two orbits of edges of length $3$ and the graph and its automorphism 
group are
$$
W_3\vee R_1\ \ \ \ \ \ \ \ \Sigma_3\times \Z/2  \ . 
$$

\medskip

\noindent {\em 2.2.5\ \ } Now consider the case of $8$ edges and $5$ 
vertices. If a $3$-Sylow subgroup $P$ acts trivially on the 
vertices then $\Gamma$ 
contains $\Theta_2$ as subgraph (because it is supposed to be $3$ - 
singular and admissible) and the valency $3$ condition together 
with the requirement that there are no separating edges in $\Gamma$ 
imply that one would need more than $8$ edges and hence there is 
no such graph. If $P$ acts nontrivially on the vertices 
then the set of edges which have one of the moving vertices as an endpoint 
must form two orbits each of length $3$. Then the two fixed edges 
must join the two fixed vertices and because these vertices must have valency 
at least $3$ they must also be endpoints of moving edges. 
However, this violates the valency condition for the moving vertices 
and the fact there are only $8$ edges in $\Gamma$.   
\medskip
 
\noindent {\em 2.2.6\ \ }Finally we consider the case of $9$ edges
and $6$ vertices.
\noindent 
Then all vertices have valency $3$ and connectivity forces 
that $\Gamma$ cannot contain $\Theta_2$ as a subgraph. This implies 
that a nontrivial graph automorphism of order $3$ cannot fix all 
vertices. Furthermore, if such an automorphism does have fixed points 
then it has precisely three, and the valency condition implies that 
we are in the case
$$
K_{3,3}\ \ \ \ \ \ \ \  \Sigma_3\wr \Z/2 \  . 
$$  

If such an automorphism has no fixed points we consider the two 
orbits of vertices both of length $3$. If there are two vertices which 
have more than one edge joining them then admissibility of the 
graph requires these vertices to be in different orbits and thus 
there are $3$ pairs of vertices with precisely two edges between them. 
In this case the graph and its automorphism group is
$$
S_0\ \ \ \ \ \ \ \ \Z/2\wr \Sigma_3  \ .
$$ 

In the remaining case there are either no edges between vertices in 
the same orbit and we obtain again $K_{3,3}$, 
or among the three orbits of edges there are two each of which forms 
a triangle joining the vertices 
in the same orbit and the third orbit of edges connects both triangles. 
The resulting graph and its automorphism group are
$$
P_1\ \ \ \ \ \ \ \ \Sigma_3\times \Z/2   \ .
$$

\section{The isotropy spectral sequence in case $p=3$}

We recall that for a discrete group $G$ and a $G$-CW-complex $X$ there is an 
``isotropy spectral sequence" \cite{Bro} converging to the cohomology 
$H^*_G(X;\FF_p)$ of the Borel construction $EG\times_GX$. It takes the form  
$$
E_1^{p,q}\cong \prod_{\overline{\sigma}\in C_p(X)}H^q(Stab(\sigma);\FF_p)
\Rightarrow H^{p+q}_{G}(X;\FF_p) 
$$
where $\overline{\sigma}$ runs through the set of $G$-orbits of $p$-cells of 
$X$ and $Stab(\sigma)$ is the stabilizer of a representative $\sigma$ 
of this orbit.   

In this section we prove Theorem 1.1 and Corollary 1.2 by 
evaluating the isotropy spectral sequence for the action of $Out(F_4)$
on the singular locus $(K_4)_s$ and we compute thus the equivariant 
mod-$3$ cohomology $H^*_{Out(F_4)}((K_4)_s;\FF_3)$. 

We note that $H^*_{Out(F_4)}(K_4, (K_4)_s;\FF_3)$ 
is isomorphic to the cohomology of the quotient 
$H^*(Out(F_4)\backslash(K_4, (K_4)_s);\FF_3)$. In particular it vanishes 
in degrees bigger than $5$ and hence the result of our computation 
agrees with $H^*(Out(F_4);\FF_3)$ in degrees bigger than $5$.

\bigskip
\subsection{The cell structure of the quotient complex $(K_4)_s/Out(F_4)$}
\label{cell structure}

The $0$-cells of this quotient are in one to one correspondence with the 
$3$-singular graphs given in Table 1. 
\medbreak  
The $1$-cells are in bijection with pairs given by a $3$-singular 
graph $\Gamma$ together with an orbit of an invariant forest with 
respect to the action of  $Aut(\Gamma)$. By going through the list 
of $3$-singular graphs we get the list of $1$-cells given in Table 2. 
The first column in this table gives the vertices of each 
cell. (It turns out that there are no loops and that with the exception 
of two edges all edges are determined by their vertices.) The second column 
describes the invariant forest in the graph which appears first 
in the name of the $1$-cell; the tree $K_{3,1}$ in this column  
is given its standard name as the complete bipartite graph 
on one block of three vertices and one block of one vertex. 
The last column gives the abstract structure of the 
isotropy group of a representatitve of the edge in $(K_4)_s$; 
this abstract group is always to be regarded as a subgroup 
of the automorphism group of the first graph (with a fixed marking), 
namely as the subgroup which leaves the given forest invariant.

\medbreak 
Similarly, Table 3 resp. Table 4 give the lists of $2$-cells 
resp. $3$-cells. Again the first column gives the vertices, 
the second column the chain of forests to be collapsed and 
the last column the automorphism group of a representing $2$-simplex 
resp. $3$-simplex in $(K_4)_s$. (The maximal trees in the chain of forests 
describing the $3$-cells are supposed to be outside the 
subgraph $\Theta_2$.)
\medbreak 
We immediately see that the quotient complex $(K_4)_s/Out(F_4)$ has $3$ 
connected components which we will call the rose component resp. 
the $\Theta_2^{1,1}$ component resp. the $K_{3,3}$ component. 
We let $K_A$ resp. $K_B$ resp. $K_C$ be the preimages of these 
components in $(K_4)_s$. Then we have a canonical isomorphism 
$$
H^*_{Out(F_4)}((K_4)_s;\FF_3)\cong H^*_{Out(F_4)}(K_A;\FF_3)
\oplus H^*_{Out(F_4)}(K_B;\FF_3)\oplus H^*_{Out(F_4)}(K_C;\FF_3)
$$

The analysis of the first two summands is quite straight forward, 
while the analysis of the last one is more delicate. 
  
\begin{table}[htp]
{\footnotesize
\setlength{\arraycolsep}{1.35mm}
$$
\begin{array}{|l|c|c|}
\hline
\mbox{Vertices of the cell }&\mbox{Invariant forest}
&\mbox{Isotropy}\\
\hline\hline
\raisebox{0.0ex}[1.2\height]{$\Theta_2\ast\ast\Theta_1$, $\Theta_3\ast R_1$}&  
\text{one of the top horizontal edges} &\Sigma_3\times \Z/2  
\\ \hline
\raisebox{0.0ex}[1.2\height]{$\Theta_2\ast\ast\Theta_1$, $\Theta_2\diamond Y$}
&  \text{one of the vertical edges} &\Sigma_3\times \Z/2  
\\ \hline
\raisebox{0.0ex}[1.2\height]{$\Theta_2\ast\ast\Theta_1$, $\Theta_3^{0,1}$}&  
\text{a top horizontal edge and a vertical edge} &\Sigma_3  
\\ \hline 
\raisebox{0.0ex}[1.2\height]{$\Theta_2\ast\ast\Theta_1$, $\Theta_4$}&  
\text{both vertical edges} &\Sigma_3\times (\Z/2)^2  
\\ \hline  
\raisebox{0.0ex}[1.2\height]{$\Theta_2\ast\ast\Theta_1$, $R_4$}&  
\text{both vertical edges and a top horizontal edge} &\Sigma_3\times \Z/2  
\\ \hline   \hline 
\raisebox{0.0ex}[1.2\height]{$\Theta_3\ast R_1$, $\Theta_3^{0,1}$}&  
\text{an edge joining the base of $R_1$ to another vertex} 
&\Sigma_3\times \Z/2  
\\ \hline  
\raisebox{0.0ex}[1.2\height]{$\Theta_3\ast R_1$, $R_4$}&  
\text{both edges joining the base of $R_1$ to the other vertices} 
&\Sigma_3\times (\Z/2)^2  
\\ \hline  \hline 
\raisebox{0.0ex}[1.2\height]{$\Theta_2\diamond Y$, $\Theta_3^{0,1}$}&  
\text{one of the vertical right hand edges} &\Sigma_3  
\\ \hline 
\raisebox{0.0ex}[1.2\height]{$\Theta_2\diamond Y$, $\Theta_4$}&  
\text{vertical left hand edges} &\Sigma_3\times \Z/2  
\\ \hline 
\raisebox{0.0ex}[1.2\height]{$\Theta_2\diamond Y$, $R_4$}&  
\text{vertical left hand edge and a vertical right hand edge} 
&\Sigma_3  
\\ \hline \hline 
\raisebox{0.0ex}[1.2\height]{$\Theta_3^{0,1}$, $R_4$}&  
\text{one of the edges of $\Theta_3$} &\Sigma_3\times \Z/2  
\\ \hline \hline 
\raisebox{0.0ex}[1.2\height]{$\Theta_4$, $R_4$}&  
\text{one of the edges of $\Theta_4$} &\Sigma_4\times \Z/2  
\\ \hline \hline 
\raisebox{0.0ex}[1.2\height]{$W_3\vee R_1$, $R_4$}&  
\text{symmetric tree around the center} &\Sigma_3\times \Z/2  
\\ \hline \hline \hline 
\raisebox{0.0ex}[1.2\height]{$K_{3,3}$, $\Theta_2\vee \Theta_2$}
&\text{$K_{3,1}$}& \Sigma_3\times\Z/2     
\\ \hline  
\raisebox{0.0ex}[1.2\height]{$K_{3,3}$, $T_1$}      
&\text{Invariant forest made of $3$ disjoint edges}& \Sigma_3\times \Z/2  
\\ \hline \hline
\raisebox{0.0ex}[1.2\height]{$P_1$, $T_1$}&
\text{the three edges joining the two triangles}  
&\Sigma_3\times\Z/2  
\\ \hline \hline 
\raisebox{0.0ex}[1.2\height]{$S_0$, $T_1$}&
\text{Invariant forest of an orbit of a ``single edge"}&\Z/2\wr \Sigma_3  
\\ \hline  
\raisebox{0.0ex}[1.2\height]{$S_0$, $T_0$}&
\text{Invariant forest of an orbit of a ``double edge"}&\Sigma_3  
\\ \hline \hline 
\raisebox{0.0ex}[1.2\height]{$\Theta_2 :\Theta_1$, $\Theta_2\vee \Theta_2$}&
\text{one of the two edges joining $\Theta_2$ to the vertical 
$\Theta_1$}&\Sigma_3\times\Z/2  
\\ \hline  
\raisebox{0.0ex}[1.2\height]{$\Theta_2 :\Theta_1$, 
$\Theta_2\vee \Theta_1\vee R_1$}&
\text{one of the two edges of the vertical $\Theta_1$}&\Sigma_3\times \Z/2  
\\ \hline  
\raisebox{0.0ex}[1.2\height]{$\Theta_2 : \Theta_1$, $\Theta_2^{0,2}$}&
\text{the union of the two edges in the previous two cases}&\Sigma_3  
\\ \hline
\raisebox{0.0ex}[1.2\height]{$\Theta_2 : \Theta_1$, $\Theta_2^{0,2}$}&
\text{both edges joining $\Theta_2$ to the vertical $\Theta_1$}&
\Sigma_3\times (\Z/2)^2 
\\ \hline \hline 
\raisebox{0.0ex}[1.2\height]{$\Theta_2\vee \Theta_1\vee R_1$, $\Theta_2^{0,2}$}&
\text{one of the edges of $\Theta_1$}&\Sigma_3\times \Z/2  
\\ \hline \hline
\raisebox{0.0ex}[1.2\height]{$\Theta_2\vee \Theta_2$, $\Theta_2^{0,2}$}&
\text{any edge}&\Sigma_3\times \Z/2 
\\ \hline 
\end{array}
$$}
\caption{\label{1-cells} $1$-cells in $(K_4)_s/Out(F_4)$}
\end{table}

\begin{table}[htp]
{\footnotesize
\setlength{\arraycolsep}{1.35mm}
$$
\begin{array}{|l|c|c|}
\hline
\mbox{Cell}&\mbox{Chain of invariant forests}&\mbox{Isotropy}\\ 
\hline\hline
\raisebox{0.0ex}[1.2\height]{$\Theta_2\ast\ast\Theta_1$, $\Theta_3\ast R_1$, 
$\Theta_3^{0,1}$}&  
\text{top horizontal edge, then a vertical edge} &\Sigma_3  
\\ \hline
\raisebox{0.0ex}[1.2\height]{$\Theta_2\ast\ast\Theta_1$, $\Theta_3\ast R_1$, 
$R_4$}&  
\text{top horizontal edge, then maximal tree outside of $\Theta_2$} &
\Sigma_3\times \Z/2  
\\ \hline
\raisebox{0.0ex}[1.2\height]{$\Theta_2\ast\ast\Theta_1$, $\Theta_2\diamond Y$, 
$\Theta_3^{0,1}$}&  
\text{a vertical edge, then top horizontal edge } &\Sigma_3
\\ \hline  
\raisebox{0.0ex}[1.2\height]{$\Theta_2\ast\ast\Theta_1$, $\Theta_2\diamond Y$, 
$\Theta_4$}&  
\text{a vertical edge, then the other vertical edge } &\Sigma_3\times \Z/2  
\\ \hline 
\raisebox{0.0ex}[1.2\height]{$\Theta_2\ast\ast\Theta_1$, $\Theta_2\diamond Y$, 
$R_4$}&  
\text{a vertical edge, then maximal tree outside of $\Theta_2$} &\Sigma_3  
\\ \hline 
\raisebox{0.0ex}[1.2\height]{$\Theta_2\ast\ast\Theta_1$, $\Theta_3^{0,1}$, 
$R_4$}&  
\text{top horizontal and vertical edge, then maximal tree outside of 
$\Theta_2$} &\Sigma_3  
\\ \hline 
\raisebox{0.0ex}[1.2\height]{$\Theta_2\ast\ast\Theta_1$, $\Theta_4$, $R_4$}&  
\text{both vertical edges, then maximal tree outside of $\Theta_2$} &
\Sigma_3\times \Z/2  
\\ \hline \hline 
\raisebox{0.0ex}[1.2\height]{$\Theta_3\ast R_1$, $\Theta_3^{0,1}$, $R_4$}&  
\text{one edge, then both edges joining $\Theta_3$ with $R_1$} &
\Sigma_3\times \Z/2 
\\ \hline \hline 
\raisebox{0.0ex}[1.2\height]{$\Theta_2\diamond Y$, $\Theta_3^{0,1}$, $R_4$}&  
\text{one right hand vertical edge, then maximal tree outside of $\Theta_2$} &
\Sigma_3  
\\ \hline  
\raisebox{0.0ex}[1.2\height]{$\Theta_2\diamond Y$, $\Theta_4$, $R_4$}&  
\text{left hand vertical edge, then maximal tree outside of $\Theta_2$} &
\Sigma_3  
\\ \hline \hline \hline 
\raisebox{0.0ex}[1.2\height]{$\Theta_2:\Theta_1$, 
$\Theta_2\vee \Theta_1\vee R_1$,  
$\Theta_2^{0,2}$}&  \text{one edge of $\Theta_1$, then add edge between 
$\Theta_2$ and $\Theta_1$} 
&\Sigma_3  
\\ \hline  
\raisebox{0.0ex}[1.2\height]{$\Theta_2:\Theta_1$, $\Theta_2\vee \Theta_2$,  
$\Theta_2^{0,2}$}&  \text{one edge between $\Theta_2$ and $\Theta_1$, then 
both edges} 
&\Sigma_3\times \Z/2  
\\ \hline  
\raisebox{0.0ex}[1.2\height]{$\Theta_2:\Theta_1$, $\Theta_2\vee \Theta_2$, 
$\Theta_2^{0,2}$}&  \text{one edge between $\Theta_2$ and $\Theta_1$, then 
add edge of $\Theta_1$ } 
&\Sigma_3  
\\ \hline 
\end{array} 
$$}
\caption{\label{2-cells} $2$-cells in $(K_4)_s/Out(F_{4})$}
\end{table}

\begin{table}[htp]
{\footnotesize
\setlength{\arraycolsep}{1.35mm}
$$
\begin{array}{|l|c|c|}
\hline
\mbox{Cell}&\mbox{Chain of invariant forests}&\mbox{Isotropy}\\ 
\hline\hline
\raisebox{0.0ex}[1.2\height]{$\Theta_2\ast\ast\Theta_1$, $\Theta_3\ast R_1$, 
$\Theta_3^{0,1}$, $R_4$}&  
\text{top horizontal edge, then add a vertical edge, then maximal tree} 
&\Sigma_3 
\\ \hline
\raisebox{0.0ex}[1.2\height]{$\Theta_2\ast\ast\Theta_1$, $\Theta_2\diamond Y$, 
$\Theta_3^{0,1}$, $R_4$}&  
\text{one vertical edge, then top horizontal edge, then maximal tree} &
\Sigma_3  
\\ \hline
\raisebox{0.0ex}[1.2\height]{$\Theta_2\ast\ast\Theta_1$, $\Theta_2\diamond Y$, 
$\Theta_4$, $R_4$}&  
\text{one vertical edge, then both vertical edges, then maximal tree} 
&\Sigma_3 \\ 
\hline\end{array}
$$}
\caption{\label{3-cells} $3$-cells in $(K_4)_s/Out(F_4)$}
\end{table}

\subsection{The rose component}
  
This component turns out to be the realization of a poset which 
is described in Figure~\ref{Fig3-1}. As a simplicial complex 
it is a one point union of the $1$-cell with vertices $W_3\vee R_1$ and $R_4$, 
and the cone (with cone point $\Theta_2\ast\ast\Theta_1$) over the pentagon 
formed by the adjacent $2$-simplices with vertices $\Theta_3\ast R_1$, 
$\Theta_3^{0,1}$, $R_4$ resp. $\Theta_2\diamond Y$, $\Theta_3^{0,1}$, $R_4$ 
resp. $\Theta_2\diamond Y$, $\Theta_4$, $R_4$. 
In particular, the rose component is contractible. Furthermore 
in the $E_1$-term of the isotropy spectral sequence for $K_A$ the 
contribution of each cell is isomorphic to $H^*(\Sigma_3;\FF_3)$ 
and all faces induce the identity via this identification 
(cf. Lemma 4.1 of \cite{GM}). 
Therefore the equivariant cohomology turns out to be 
$$
H^*_{Out(F_4)}(K_A;\FF_3)\cong H^*(\Sigma_3;\FF_3)\cong H^*(G_R;\FF_3) \ .
$$

\begin{figure}[ht!]
\begin{footnotesize}
\begin{center}
\includegraphics[width=0.70\hsize]{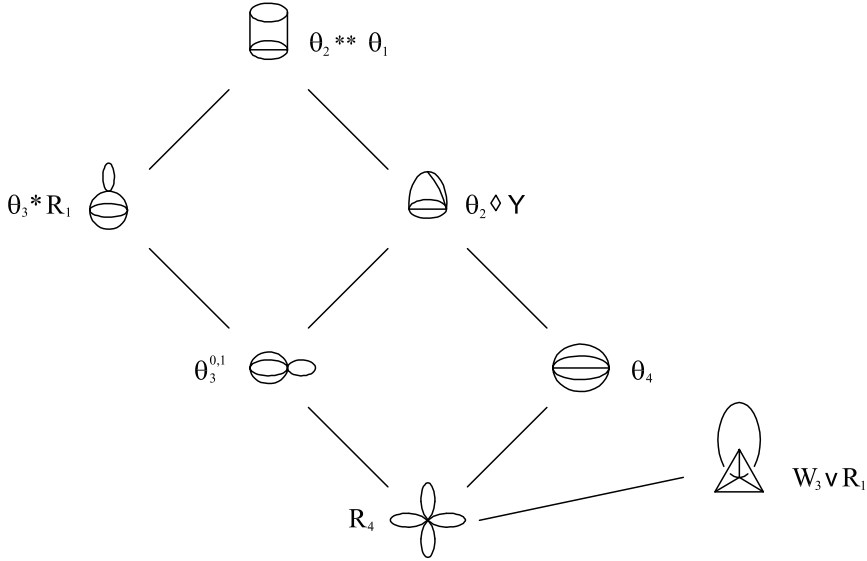}
\caption{\label{Fig3-1}\footnotesize  The rose component of
$Out(F_4)\setminus (K_4)_3$.}
\end{center}
\end{footnotesize}
\end{figure}

\subsection{The $\Theta_2^{1,1}$ component}
  
This component is even simpler; it consists of a 
single point and therefore we get 
$$
H^*_{Out(F_4)}(K_B;\FF_3)\cong H^*(\Sigma_3;\FF_3)\cong H^*(G_{1,1};\FF_3)\ . 
$$

\subsection{The $K_{3,3}$ component}
  
The geometry of the remaining component is as follows:
there is a ``critical edge'' joining $K_{3,3}$ to $\Theta_2\vee\Theta_2$,
there is a ``free part'' (which in Figure 4 is to the right of the critical 
edge 
and in which the $3$-Sylow subgroups of the automorphism groups of the graphs 
act freely on the graph), and there is the  ``fixed part'' which is attached to 
$\Theta_2\vee \Theta_2$ (on which the $3$-Sylow subgroups of the automorphism 
groups 
of the graphs fix at least one vertex of the graph),   
see Figure~\ref{Fig3-2}.
 
\bigbreak 
 
\begin{figure}[ht!]
\begin{footnotesize}
\begin{center}
\includegraphics[width=0.70\hsize]{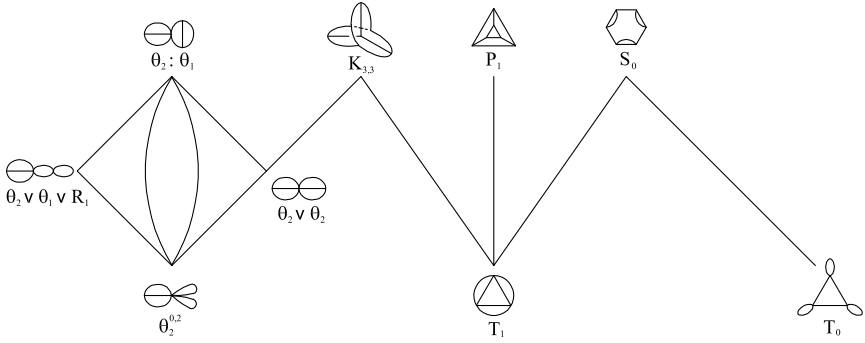}
\caption{\label{Fig3-2}\footnotesize  The $K_{3,3}$ component of 
$Out(F_4)\backslash(K_4)_s$}
\end{center}
\end{footnotesize}
\end{figure}
 
We note that in this figure there are two triangles both of whose vertices are 
$\Theta_2:\Theta_1$, $\Theta_2\vee \Theta_2$ and $\Theta_2^{0,2}$. 

\bigbreak

In the $E_1$-term of the isotropy spectral sequence for $K_C$ the contribution 
of each cell except those on the  ``critical edge" between $K_{3,3}$ and 
$\Theta_2\vee \Theta_2$ is isomorphic to $H^*(\Sigma_3;\FF_3)$ 
and all faces except those involving the critical edge induce the identity 
via this identification (cf. Lemma 4.1 of \cite{GM} again). This together 
with the observation that the critical edge is a deformation retract of the 
quotient 
$Out(F_4)\backslash K_C$ implies that the inclusion of the preimage 
$K_4^e$ of the critical edge into $K_C$ 
induces an isomorphism in equivariant cohomology 
$$
H^*_{Out(F_4)}(K_C;\FF_3)\to H^*_{Out(F_4)}(K_4^e;\FF_3)\ . 
$$

So it remains to calculate  $H^*_{Out(F_4)}(K_4^e;\FF_3)$. Because 
$K_4^e$ is a $1$-dimensional with quotient an edge the following lemma 
finishes off the proof of Theorem 1.1.

\begin{lemma}\label{amalgam} 
Let $G$ be a discrete group and $X$ a $1$-dimensional 
$G$-$CW$-complex with fundamental domain a segment. 
Let $G_1$ and $G_2$ be the isotropy groups of the vertices of the segment 
and $H$ be the isotropy group of the segment itself and let $A$ be any abelian 
group. Then there is a canonical isomorphism 
$$
H^*_G(X,A)\cong H^*(G_1*_HG_2,A) \ .
$$
\end{lemma}

\begin{proof} The inclusions of $G_1$, $G_2$ and $H$ into $G$ 
determine a homomorphism from $G_1*_HG_2$ to $G$. Furthermore the 
tree $T$ associated to the amalgamated product admits a 
$G_1*_HG_2$-equivariant map to $X$ such that the induced map on spectral 
sequences converging to $H^*_G(X,A)$ resp. to 
$H^*_{G_1*_HG_2}(T,A)\cong H^*(G_1*_HG_2,A)$ is an isomorphism 
on $E_1$-terms. $\square$
\end{proof}

\subsection{The critical edge and the proof of Corollary 1.2}
 
\noindent The isotropy groups along this edge are as follows:
for $K_{3,3}$ and for $\Theta_2\vee\Theta_2$
we get both times $\Sigma_3\wr \Z/2$
while for the edge we get $\Sigma_3\times \Z/2$.
 
To understand the two inclusions of $\Sigma_3\times \Z/2$ into 
$\Sigma_3\wr \Z/2$
we observe that the edge is obtained by collapsing an invariant
$K_{3,1}$ tree, hence the isotropy group of the edge embeds into that of
$K_{3,3}$ via the embedding of $\Sigma_3\times \Z/2$ into
$\Sigma_3\times\Sigma_3$.
On the other hand it embeds into the isotropy group of
$\Theta_2\vee\Theta_2$ via the ``diagonal embedding''. 
 
The cohomology of the isotropy groups considered as abstract groups
is well known and given for the convenience of the reader 
in the following proposition.

\begin{proposition} 
a)
$$
H^*(\Sigma_3\times \Z/2;\FF_3)\cong H^*({\Sigma_3};\FF_3)\cong 
\FF_3[a_4]\otimes \Lambda(b_3)
$$
where the indices give the dimensions of the elements.
 
\medskip
b)
$$
H^*(\Sigma_3\wr \Z/2;\FF_3)\cong H^*(\Sigma_3\times\Sigma_3;\FF_3)^{\Z/2}
\cong (\FF_3[c_{4,1},c_{4,2}]\otimes \Lambda(d_{3,1},d_{3,2}))^{\Z/2}
$$
where the first index gives the dimension and the second index refers 
to the first resp. second copy of $\Sigma_3$. Consequently
$$
H^*(\Sigma_3\wr \Z/2;\FF_3)\cong \FF_3[c_{4},c_8]\otimes \Lambda(d_{3}, d_7)
$$
where
$$
c_4=c_{4,1}+c_{4,2}  \ \ c_8=(c_{4,1}-c_{4,2})^2
$$
$$
d_3=d_{3,1}+d_{3,2}  \ \  d_7=(c_{4,1}-c_{4,2})(d_{3,1}-d_{3,2}). \ \ \ \square
$$
\end{proposition}
 
Next we turn towards the description of the restriction maps.
Let $\alpha$ denote the map in mod $3$ cohomology induced by the inclusion of
$\Sigma_3\times\Z/2$ into the isotropy group
$G_K$ and $\beta$ be the map induced
by the inclusion of $\Sigma_3\times\Z/2$ into $G_2$.

We change notation and write
$$
H^*(G_K;\FF_3)\cong\ \FF_3[x_4,x_8]\otimes \Lambda(u_3,u_7)
$$
$$
H^*(G_2;\FF_3)\cong
\FF_3[y_4,y_8]\otimes \Lambda(v_3,v_7)
$$
and
$$
H^*(\Sigma_3\times\Z/2;\FF_3)\cong \FF_3[z_4]\otimes \Lambda(w_3)\ .
$$
  
Then the effect of the two restriction maps is as follows:
$$
\alpha(x_4)=z_4, \ \ \alpha(x_8)=z_4^{2}, \ \ \alpha(u_3)=w_3, 
\ \ \alpha(u_7)=z_4w_3
$$
and
$$
\beta(y_4)=2z_4, \ \ \beta(y_8)=0, \ \ \beta(v_3)=2w_3, \ \ \beta(v_7)=0\ .
$$

\begin{proposition} 

a) The inclusions of $K_4^e$ into the $Out(F_4)$-
orbits of $K_{3,3}$ and $\Theta_2\vee\Theta_2$
induce an isomorphism 
$$
H^*_{Out(F_4)}(K_4^e;\FF_3)\cong Eq(\a,\b)
$$
between $H^*_{Out(F_4)}(K_4^e;\FF_3)$
and the subalgebra of 
$\FF_3[x_4,x_8]\otimes \Lambda(u_3,u_7)\times
\FF_3[y_4,y_8]\otimes \Lambda(v_3,v_7)$
equalized by the maps $\alpha$ and $\beta$.
 
b) This equalizer contains the tensor product of the
polynomial subalgebra generated by the elements
$r_4=(x_4,2y_4)$ and $r_8=(x_8,y_4^2+y_8)$ with the exterior
algebra generated by the element $s_3=(u_3,2v_3)$,
and as a module over this tensor product it is free on generators
$1,t_7,\ww{t_7}, t_8$, i.e.we can write
$$
H^*_{Out(F_4)}(X_4^e;\FF_3)\cong
\FF_3[r_4,r_8]\otimes \Lambda(s_3)\{1,t_7,\ww{t_7}, t_8\}
$$
with $1=(1,1)$, $t_7=(0,v_7)$, $\ww{t_7}=(u_7,y_4v_3)$, $t_8=(0,y_8)$.

c) The additional multiplicative relations in this algebra
are given as  
$$
t_7^2=\ww{t_7}{ }^{2}=0, \  t_8^2=(r_8-r_4^2)t_8, 
\ \ww{t_7}t_7=r_4s_3t_7,  \  t_8t_7=(r_8-r_4^2)t_7, \ 
t_8\ww{t_7}=r_4s_3t_8 \ .
$$ 
\end{proposition}  
 
\begin{proof} a) It is clear that $\alpha$ is onto and this implies 
that $H^*_{Out(F_4)}(K_4^e;\FF_3)$  
is given as the equalizer of $\alpha$ and $\beta$.  

b) It is straightforward to check that all of the elements
$r_4$, $r_8$, $s_3$ and $t_i$ are contained in the equalizer
and that the subalgebra generated by $r_4,r_8$
and $s_3$ has structure as claimed. It is also easy to check
that the remaining elements
$1,t_7,\ww{t_7}$ and $t_8$ are linearly independent over this subalgebra.
Furthermore we have the following equation
for the Euler Poincar\'e series $\chi$ of the equalizer:
$$
\chi + \frac{1+t^3}{1-t^4}=2\frac{(1+t^3)(1+t^7)}{(1-t^4)(1-t^8)} \ .
$$
From this we get
$$
\chi=\frac{(1+t^3)(1+2t^7+t^8)}{(1-t^4)(1-t^8)} 
$$
and hence the result. 

c) The additional multiplicative relations in the algebra
$H^*_{Out(F_4)}(X_4^e;\FF_3)$
can be easily determined by considering it as subalgebra of
$\FF_3[x_4,x_8]\otimes \Lambda(u_3,u_7)\times 
\FF_3[y_4,y_8]\otimes \Lambda(v_3,v_7)$. 
$\square$
\end{proof}
 
\bigskip\bigskip
 
\section{Elementary abelian p-subgroups in $Out(F_{2(p-1)})$ for $p>3$}

In this section we determine the conjugacy classes 
of elementary abelian $p$-subgroups of the group $Out(F_{2(p-1)})$. 
The strategy is as follows. If $G$ is a finite subgroup then 
by results of Culler \cite{Cu} and Zimmermann \cite{Z} there exists a 
finite graph $\Gamma$ with $\pi_1(\Gamma)\cong F_n$ and a subgroup 
$G'$ of $Aut(\Gamma)$ such that $G'$ gets identified with 
$G$ via the induced outer action on $\pi_1(\Gamma)\cong F_n$.  
In the sequel we will call such a graph a $G$-graph. 
Changing the marking of a $G$-graph amounts to changing $G$ 
within its conjugacy class.

\subsection{Reduced $\Z/p$-graphs of rank $n=2(p-1)$}

\smallbreak\noindent
We say that $\Gamma$ is $G$-reduced, 
if $\Gamma$ does not contain any $G$-invariant forest. By collapsing 
tress in invariant forests to a point we may assume that the graph $\Gamma$ 
in the result of Culler and Zimmermann is reduced. We thus get an upper 
bound for the conjugacy classes in terms of isomorphism classes of reduced 
graphs of rank $n$ with a $\Z/p$-symmetry.  
\medbreak 
\noindent 
We thus proceed to classify $\Z/p$-reduced graphs of rank $n=2(p-1)$. 
\medbreak
\noindent
a) If there is a vertex $v_1$ which is fixed by $\Z/p$ and if $e$ is any 
edge joining this vertex to any distinct vertex $v_2$ then $v_2$ 
has to be also fixed (otherwise the orbit of this edge gives an 
invariant forest). Therefore if there is one vertex which is fixed 
then all vertices will be fixed. 
\medbreak
\noindent
b) If there is only one fixed vertex then the graph has to be 
the rose $R_{2(p-1)}$ and $\Z/p$ acts transitively on $p$ 
of the edges of the rose and fixes the others. We choose a marking 
so that we obtain an isomorphic subgroup in $Out(F_n)$ which we 
denote $E_R$.  
\medbreak 
\noindent 
c) If we have two fixed vertices then 
any edge between them cannot be fixed because otherwise it would 
define an invariant forest. Therefore $\Gamma$ contains $\Theta_{p-1}$ 
and there cannot be any other edge between the two vertices because they 
would have to be fixed and thus there would be an invariant forest. 
Consequently the graph has to be isomorphic to $\Theta_{p-1}^{s,t}$ 
with $s+t=p-1$ and $0\leq s\leq \frac{p-1}{2}$. After having
chosen a marking, we get a subgroup of $\Z/p\cong E_{s,p-1-s}$ of 
$Out(F_n)$. 
\medbreak\noindent 
d) If there are three fixed vertices then $\Gamma$ has to be isomorphic to 
$\Theta_{p-1} \vee \Theta_{p-1}$ and $\Z/p$ acts non-trivially on both 
$\Theta_{p-1}$'s. In this case the action of $\Z/p$ can be extended to an 
action of $\Z/p\times \Z/p$ with the left resp. right hand factor 
$\Z/p$ acting on the left resp. right hand copy of $\Theta_{p-1}$. 
After having chosen a marking, we get a subgroup $\Z/p\times\Z/p\cong E_2$ 
of $Out(F_n)$. Its diagonal will be denoted $\Delta(E_2)$. 
\medbreak\noindent 
e) Clearly there are no reduced graphs of rank $2(p-1)$ which have more than 
$3$ fixed vertices. 
\medbreak\noindent  
f) If $\Z/p$ acts without fixed vertex then there are also no fixed edges and 
we get for the Euler characteristic 
$$
1-2(p-1)=\chi(\Gamma)\equiv 0 \mod (p) 
$$
and hence $p=3$. 

\medbreak
\noindent 
We have thus proved the following result (cf. Proposition 4.2 of \cite{J}). 

\begin{proposition}\label{reduced graphs} Let $p>3$ be a prime and $\Gamma$ 
be a reduced $\Z/p$-graph of rank $n=2(p-1)$.  
Then $\Gamma$ is isomorphic to one of the graphs $R_n$, 
$\Theta_{p-1}^{s,t}$ with $s+t=n$ and $0\leq s\leq \frac{p-1}{2}$, 
or $\Theta_{p-1}\vee \Theta_{p-1}$ (with diagonal action of $\Z/p$).  
\ \ \ $\square$ 
\end{proposition}

\subsection{Nielsen transformations of $G$-graphs and 
conjugacy classes of finite subgroups of $Out(F_n)$} 

In order to distinguish conjugacy classes of elementary abelian 
$p$-subgroups 
we make use of Krstic's theory of equivariant Nielsen transformations \cite{K} 
which we will also use to determine most of 
the centralizers and normalizers of these elementary abelian subgroups.

\begin{definition} 
Let $G$ be a finite subgroup of $Out(F_n)$ and $\Gamma$ be a reduced 
$G$-graph of rank $n$ with vertex set $V$ and edge set $E$. Suppose 
\smallbreak
$\bullet$ $e_1$ and $e_2$ are edges such that $e_2$ is neither in 
the orbit of $e_1$ nor of its opposite $\sigma({e_1})$, 
\smallbreak 
$\bullet$ $e_1$ and $e_2$ have the same terminal points, $t(e_1)=t(e_2)$,  
\smallbreak 
$\bullet$ the stabilizer of $e_1$ in $G$ is contained in that of $e_2$.
\medbreak  
Then there is a unique graph  $\Gamma'$ 
with $V'=V$, $E'=E$, the same $G$-action as on $\Gamma$, $\sigma'=\sigma$, 
$t'(e)=e$ if $e$ is not in the orbit of $e_1$, and $t'(ge_1)=t(\sigma(ge_2))$ 
for every $g\in G$. This graph is denoted $<e_1,e_2>\Gamma$ 
and the assignment $\Gamma\mapsto <e_1,e_2>\Gamma$   
is called a Nielsen transformation.  
\end{definition}

In the situation of this definition there is a $G$-equivariant isomorphism, 
also called a $G$-equivariant Nielsen isomorphism, 
$<e_1,e_2>:\Pi(\Gamma)\to \Pi(\Gamma')$ of fundamental groupoids;  
it is determined by its map on edges by $<e_1,e_2>(e)=e$ if $e$ is 
neither in the $G$-orbit of $e_1$ or $\sigma(e_1)$, and 
$<e_1,e_2>(ge_1)=g(e_1e_2)$ 
for every $g\in G$.

\begin{proposition} Let $p>3$ be a prime.  

\noindent
a) The groups $E_R$, $E_{s,p-1-s}$ for 
$0\leq s\leq \frac{p-1}{2}$, $E_2$ and the diagonal $\Delta(E_2)$ 
of $E_2$ are pairwise non-conjugate, and any 
elementary abelian $p$-subgroups of $Out(F_{2(p-1)})$ is 
conjugate to one of them.
\smallbreak\noindent  
b) $E_{0,p-1}$ is conjugate to the subgroup $\Z/p$ of $E_2$ which acts 
only on the left hand $\Theta_{p-1}$ in $\Theta_{p-1}\vee \Theta_{p-1}$. 
\smallbreak\noindent 
c) Neither $E_R$ nor any of the $E_{s,p-1-s}$ with $1\leq s\leq \frac{p-1}{2}$ 
is conjugate to a subgroup of $E_2$. 
\end{proposition}   

\begin{proof}  Proposition \ref{reduced graphs} and the realization 
result of Culler and Zimmermann show that every elementary abelian 
$p$-subgroup of $Out(F_{2(p-1)})$ is conjugate to one in the given list. 
(Here we have implicitly used that the action of the group 
$Aut(\Theta_{p-1}\vee \Theta_{p-1})$ on its subgroups of 
order $p$ has just two orbits, that of $\Delta(E_2)$ and that of 
the subgroup fixing the right hand $\Theta_{p-1}$. The latter one can be 
omitted from our list because the corresponding graph is not reduced.)

Furthermore, if $G$ is a finite subgroup of $Out(F_{2(p-1)})$ which is  
realized by reduced $G$-graphs $\Gamma_1$ and $\Gamma_2$ then 
by Theorem 2 of \cite{K} there is a sequence of $G$-equivariant 
Nielsen transformations from $\Gamma_1$ to a graph which is $G$-equivariantly 
isomorphic to $\Gamma_2$.  
\smallbreak\noindent 
a) Inspection shows that any $E_R$- (resp. $\Delta(E_2)$- resp. $E_{s,p-1-s}$-) 
equivariant Nielsen transformation starting in $R_n$ (resp. 
$\Theta_{p-1}\vee \Theta_{p-1}$ resp. $\Theta_{p-1}^{s,t}$)  
ends up in a graph which is isomorphic to the initial graph. 
This shows, in particular, that none of subgroups $E_R$, $E_{s,p-1-s}$ and 
$\Delta(E_2)$ can be conjugate.   
\smallbreak\noindent 
b) The only subgroups of $E_2$ for which the graph 
$\Theta_{p-1}\vee \Theta_{p-1}$ 
is not reduced are those which act nontrivially only on one of the 
$\Theta_{p-1}$'s in $\Theta_{p-1}\vee \Theta_{p-1}$. By collapsing one of the 
fixed edges in the other $\Theta_{p-1}$ we pass to $\Theta_{p-1}^{0,p-1}$ 
and this implies that $E_{0,p-1}$ is conjugate to one of these  
subgroups of $E_2$. 
\smallbreak\noindent 
c) For all other subgroups of $E_2$ the graph $\Theta_{p-1}\vee \Theta_{p-1}$ 
is reduced and by using Theorem 2 of \cite{K} once more we see 
that neither $E_R$ 
nor $E_{s,p-1-s}$ for $1\leq s\leq \frac{p-1}{2}$ is conjugate to a subgroup of 
$E_2$. $\square$ 
\end{proof}

\section{Normalizers of elementary abelian p-subgroups}

Equivariant Nielsen transformations of graphs will also be used in the analysis 
of the normalizers of the elementary abelian subgroups of $Out(F_n)$. 
We will thus start this section by recalling more results of \cite{K}. 
\bigbreak 

Given a reduced $G$-graph $\Gamma$ of rank $n$ 
Krstic establishes in Corollary 2 
together with Proposition 2 an exact sequence of groups 
$$
1\to Inn_G(\Pi(\Gamma))\to Aut_G(\Pi(\Gamma))\to C_{Out(F_n)}(G)\to 1 \ . 
$$
Here $Aut_G(\Pi(\Gamma))$ resp. $Inn_G(\Pi(\Gamma))$ denote the $G$-equivariant
automorphisms resp. inner automorphisms of the fundamental groupoid of 
$\Gamma$. We recall that an endomorphism $J$ of a fundamental groupoid 
$\Pi$ is called inner 
if there is a collection of paths $\l_v, v\in V$, such that $t(\sigma(\l_v))=v$ 
for every $v\in V$ and $J(\a)=\sigma(\l_{t(\sigma(\a))})\a\l_{t\a}$ 
for every path $\a\in \Pi$. For a general $G$-graph an inner endomorphism 
need not be an automorphism. However, if $\Gamma$ is $G$-reduced, 
then any inner endomorphism is an isomorphism (cf. Proposition 3 of \cite{K}). 

Furthermore, the analysis of $Aut_G(\Pi(\Gamma))$
is made possible by Proposition 4 in Krstic \cite{K} which says 
that each $G$-equivariant automorphisms of $\Pi(\Gamma)$ 
can be written as a composition of $G$-equivariant Nielsen isomorphisms 
followed by a $G$-equivariant graph isomorphism. We will use this in order 
to determine the normalizers of all elementary abelian 
$p$-subgroups of $OutF_{2(p-1)}$, $p>3$, except that of $\Delta(E_2)$.  
Our discussion is very close to that of section 5 in \cite{J} where the same 
ideas are used to determine normalizers of elementary abelian subgroups 
of $Aut(F_{2(p-1)})$. Therefore our discussion will be fairly brief 
and the reader who wants to see more details may want 
to have a look at \cite{J}.

\subsection{$\Gamma=R_{2(p-1)}$}

\noindent 
We begin by constructing homomorphisms (cf. \cite{J})  
\begin{eqnarray} 
F_{p-2}\times F_{p-2}\to Aut_{E_R}(\Pi(R_{2(p-1)})),& \ \ (v,w) &\mapsto 
(y_i\mapsto y_i, x_i\mapsto v^{-1}x_iw)\nonumber \\
Aut(F_{p-2})\to Aut_{E_R}(\Pi(R_{2(p-1)})), &\ \ \alpha &\mapsto 
(y_i\mapsto \alpha(y_i), x_i\mapsto x_i)\nonumber \\
\Z/p\to Aut_{E_R}(\Pi(R_{2(p-1)})), & \ \ \sigma & \mapsto 
(y_i\mapsto y_i, x_i\mapsto x_{i+1})\nonumber \\
\Z/2\to Aut_{E_R}(\Pi R_{2(p-1)}),& \ \ \tau & \mapsto 
( y_i\mapsto y_i, x_i\mapsto x_i^{-1})\ . \nonumber 
\end{eqnarray}
Here the $x_i$ are the edges of $R_n$ which get cyclically moved 
by $\Z/p\cong E_R$, the $y_i$ are the fixed edges of $R_n$, $v$ and $w$ 
are words in the $y_i$ and their inverses, and $\sigma$ is 
a suitable generator of $E_R$.  
These maps determine a homomorphism 
$$
\psi_R: \Z/p\times ((F_{p-2}\times F_{p-2})\rtimes (Aut(F_{p-2})\times \Z/2))
\to C_{Aut(F_{2(p-1)})}(E_R)
$$ 
in which the action of $Aut(F_{p-2})$ on $F_{p-2}\times F_{p-2}$ is the 
canonical diagonal action, while
$\tau$ acts via $\tau(v,w)\tau^{-1}=(w,v)$. This homomorphism 
is surjective by Proposition 4 of \cite{K} and arguing with reduced words in 
free groups shows that it is also injective.  
 
The elements $(v^{-1},v^{-1},c_v,1)$ (where $c_v$ denotes conjugation 
$y_i\mapsto vy_iv^{-1}$) of the semidirect product 
correspond to the inner automorphisms $z\mapsto vzv^{-1}$. 
Furthermore, we have the following identity 
$$
\tau((v,1,\alpha,1)\tau^{-1}=(1,v,\alpha,1)
=(v,v,c_{v^{-1}},1)(v^{-1},1,c_v\alpha,1)
$$ 
in $(F_{p-2}\times F_{p-2})\rtimes (Aut(F_{p-2})\times \Z/2)$.  
This implies that there is an induced isomorphism
$$
\varphi_R:\Z/p\times (F_{p-2}\rtimes Aut(F_{p-2}))\rtimes\Z/2\Z
\to C_{Out(F_{2(p-1)})}(E_R)\ . 
$$
In this case the action of $\Z/2$ on $F_{p-2}\rtimes Aut(F_{p-2})$ 
is given by $\tau(x,\alpha)=(x^{-1},c_x\alpha)$. 
So we have already proved the first part of the following result.

\begin{proposition}\label{rose-normalizer} a) There is an isomorphism  
$$
\varphi_R: \Z/p\times (F_{p-2}\rtimes Aut(F_{p-2}))\rtimes\Z/2\Z
\to C_{Out(F_{2(p-1)})}(E_R) \ . 
$$ 
b) 
$\varphi_R$ extends to an isomorphism 
$$
\widetilde{\varphi_R}: N_{\Sigma}(\Z/p)\times (F_{p-2}\rtimes Aut(F_{p-2}))
\rtimes\Z/2\Z \to N_{Out(F_{2(p-1)})}(E_R)\ . 
$$
\end{proposition}
\bigskip
 
\begin{proof} Part (a) has already been proved. For (b) we note that 
there is an exact sequence 
$$
1\to C_{Out(F_2(p-1))}(E_R)\to N_{Out(F_2(p-1)}(E_R)\to Aut(E_R)\to 1 \ .
$$ 
The normalizer does indeed surject to $Aut(E_R)$ because graph automorphisms 
induce a spliting of this sequence. In fact, the homomorphism  
$\psi_R$ can be extended to a homomorphism 
$$
\widetilde{\psi_R}: N_{\Sigma}(\Z/p)\times (F_{p-2}\times F_{p-2})
\rtimes (Aut(F_{p-2})\times \Z/2)\to N_{Aut(F_{2(p-1)})}(E_R)
$$ 
which can easily be seen to be an isomorphism and which induces 
the isomorphism $\widetilde{\varphi_R}$. $\square$ 
\end{proof}

\subsection{$\Gamma=\Theta_{p-1}^{s,p-1-s}$}

\noindent 
To simplify notation we let $t=p-1-s$. 
Again we begin by constructing homomorphisms (cf. \cite{J})  
\begin{eqnarray} 
F_s\times F_t\to Aut_{E_{s,t}}(\Pi(\Theta_{p-1}^{s,t})),
& \ \ (v,w) &\mapsto 
(a_i\mapsto a_i, b_i\mapsto v^{-1}b_iw, c_i\mapsto c_i)\nonumber \\
Aut(F_{s})\times Aut(F_t)\to Aut_{E_{s,t}}(\Pi(\Theta_{p-1}^{s,t})), 
&\ \ (\alpha,\beta) &\mapsto 
(a_i\mapsto \alpha(a_i), b_i\mapsto b_i, c_i\mapsto \beta(c_i))\nonumber \\
\Z/p\to Aut_{E_{s,t}}(\Pi(\Theta_{p-1}^{s,t})), & \ \ \sigma & \mapsto 
(a_i\mapsto a_i, b_i\mapsto b_{i+1},c_i\mapsto c_i)\ .\nonumber 
\end{eqnarray}
Here the $a_i$ denote the fixed edges attached to the left hand vertex of 
$\Theta_{p-1}$, the $b_i$ are the edges of $\Theta_{p-1}$ which get cyclically 
moved by $E_{s,p-s-1}$ and have the right hand vertex as their terminal point, 
the $c_i$ are the fixed edges attached to 
the right hand vertex of $\Theta_{p-1}$, $v$ resp. $w$ are words in the 
$a_i$ resp. $c_i$ and their inverses, and $\sigma$ is a suitable generator of 
$\Z/p\cong E_{s,t}$.

These homomorphisms determine a homomorphism 
$$
\psi_{s,t}: \Z/p\times (F_s\rtimes Aut(F_s))
\times (F_{t}\rtimes Aut(F_{t}))\to 
Aut_{E_{s,t}}^*(\Pi(\Theta_{p-1}^{s,t})) 
$$
where $Aut_{E_{s,t}}^*(\Theta_{p-1}^{s,t})$ denotes the equivariant 
automorphisms of $\Theta_{p-1}^{s,t}$ which fix both vertices. 
In fact, this homomorphism is surjective by Proposition 4 of \cite{K}, 
and arguing with reduced words in free groups shows that it is also  
injective.

If $s\neq t=p-1-s$ we find 
$Aut_{E_{s,t}}(\Pi(\Theta_{p-1}^{s,t}))=Aut_{E_{s,t}}^*(\Pi(\Theta_{p-1}^{s,t}))$, 
and if $s=t$ the group $Aut_{E_{s,t}}^*(\Pi(\Theta_{p-1}^{s,t}))$ 
is of index $2$ in $Aut_{E_{s,t}}^*(\Pi(\Theta_{p-1}^{s,t}))$, 
and there is an obvious extension of 
$\psi_{s,s}$ to an isomorphism 
$$
\psi'_{s,s}: \Z/p\times 
\big((F_s\rtimes Aut(F_s)\big)\wr \Z/2\to 
Aut_{E_{s,s}}(\Pi(\Theta_{p-1}^{s,s})) \ . 
$$
In order to get $C_{Out_{F_{s,t}}(E_{s,p-1-s}})$ we 
need to quotient out the group of equivariant inner automorphisms of 
$\Pi(\Theta_{p-1}^{s,t})$. Any inner automorphism 
is given by two paths $\l_1$ resp. $\l_2$ terminating in the two edges 
$v_1$ resp. $v_2$ of $\Theta_{p-1}^{s,t}$. Equivariance requires these 
paths to be fixed under the action of $E_{s,t}$. Therefore 
$\l_1$ can be identified with a word $v$ in the $a_i$ and their inverses and 
$\l_2$ can be identified with a word $w$ in the $c_i$ and their inverses, and 
it follows that the inner automorphism determined by $\l_1$ and 
$\l_2$ corresponds to the tuple $(1,v^{-1},c_v,w^{-1},c_w)$ in 
$\Z/p\times (F_s\rtimes Aut(F_s))\times (F_{p-1-s}\rtimes Aut(F_{p-1-s}))$. 
Passing to the quotient by the inner automorphism gives the first half of the 
following result. The second half is proved as before in the case of the rose. 

\begin{proposition}\label{theta-normalizers} 
a) For $s\neq \frac{p-1}{2}$ the isomorphism 
$\psi_{s,p-1-s}$ induces an isomorphism  
$$
\varphi_{s,p-1-s}: 
\Z/p\times Aut(F_{s})\times Aut(F_{p-1-s})\to 
C_{Out(F_{2(p-1)})}(E_{s,p-1-s}) \ . 
$$ 
while for $s=\frac{p-1}{2}$ the isomorphism $\psi'_{s,s}$ 
induces an isomorphism 
$$
\varphi_{s,s}': 
\Z/p\times (Aut(F_{s})\wr \Z/2)\to C_{Out(F_{2(p-1)})}(E_{s,s}) \ . 
$$ 

b) For $s\neq \frac{p-1}{2}$ $\varphi_{s,p-1-s}$ extends 
to an isomorphism 
$$
\widetilde{\varphi}_{s,p-1-s}: N_{\Sigma}(\Z/p)\times 
Aut(F_s)\times Aut(F_{p-1-s})
\to N_{Out(F_{2(p-1)})}(E_{s,p-1-s})\ . 
$$ 
while for $s=\frac{p-1}{2}$ the isomorphism $\varphi_{s,s}$ 
extends to an isomorphism 
$$
\widetilde{\varphi}_{s,s}': 
N_{\Sigma_p}(\Z/p)\times (Aut(F_{s})\wr \Z/2)\to C_{Out(F_{2(p-1)})}(E_{s,s}) \ . 
\square
$$ 
\end{proposition}

\subsection{$\Gamma=\Theta_{p-1}\vee\Theta_{p-1}$}
 
In this case we have to look at $Aut_{E_2}(\Pi(\Theta_{p-1}\vee\Theta_{p-1}))$.
There are no non-trivial Nielsen transformations in this case and therefore the
centralizer resp. normalizer is given completely 
in terms of graph automorphisms 
and therefore has the following form (cf. \cite{J}).

\begin{proposition}\label{E2-normalizer}
\begin{eqnarray} 
C_{Out(F_{2(p-1)})}(E_2)&\cong &\Z/p\times\Z/p \nonumber \\
N_{Out(F_{2(p-1)})}(E_2)&\cong & N_{\Sigma}(\Z/p)
\wr \Z/2 \ \ \ \ \ \ \square\nonumber 
\end{eqnarray}
\end{proposition} 
\bigbreak 

It remains to determine the structure of the normalizer of 
$\Delta(E_2)$, which we will abbreviate as in the introduction 
by $N_{\Delta}$ in order to simplify notation. 
In this case we prefer to use information 
of the fixed point space $(K_{2(p-1)})_s^{\Delta(E_2)}$ rather 
than working with Krstic's method.  We immediately note that this 
fixed point space is equipped with an action of $N_{\Delta}$ and 
it contains the graph $\Theta_{p-1}\vee \Theta_{p-1}$ as well as 
the bipartite graph $K_{p,3}$ with markings which are compatible 
with respect to the collapse of the invariant tree $K_{p,1}$ inside $K_{p,3}$. 
In the following result we fix such a marking.

\begin{proposition}\label{Delta(E_2)} Let $p\geq 3$ be a prime. 
\smallbreak 
a) The fixed point space $(K_{2(p-1)})_s^{\Delta(E_2)}$ is a tree and 
the edge $e$ determined by the collapse of $K_{p,1}$ in $K_{p,3}$ 
is a fundamental domain for the action of $N_{\Delta}$ on it. 

\smallbreak
b)
There is an isomorphism 
\begin{eqnarray}
N_{\Delta}&\cong 
& Stab_{N_{\Delta}}(\Theta_{p-1}\vee\Theta_{p-1})
*_{Stab_{N_{\Delta}}(e)}Stab_{N_{\Delta}}(K_{p,3})\nonumber \\
&\cong &
((\Z/p\times\Z/p)\rtimes(Aut(\Z/p)\times\Z/2))
*_{N_{\Sigma_p}(\Z/p)\times \Z/2}{(N_{\Sigma_p}(\Z/p)\times \Sigma_3)}  
\nonumber 
\end{eqnarray}
where $Stab_{N_{\Delta}}(\Gamma)$ denotes the stabilizer of $\Gamma$ 
with respect to the action of $N_{\Delta}$ and 
the action of $Aut(\Z/p)$ on $\Z/p\times\Z/p$ 
is given by the canonical diagonal action and that by $\Z/2$ by permuting 
the factors. 
\end{proposition}

\begin{proof} Part (b) is an immediate consequence of part (a). 
\smallbreak  
For (a) we note that $K_{2(p-1)}^{\Delta(E_2)}$ is contractible 
\cite{KV}. By the general theory of groups acting on trees 
it is therefore enough to show that this space is a 
one-dimensional complex and the edge $e$ is a fundamental domain for the 
action of $\Delta(E_2)$.

This follows immediately as soon as we have shown 
that up to isomorphism there is only one 
$\Delta(E_2)$-admissible graph containing a non-trivial 
$\Delta(E_2)$-invariant forest so that after collapsing this forest 
we get $\Theta_{p-1}\vee \Theta_{p-1}$ with the given action of $\Delta(E_2)$. 
In fact, if $\Gamma$ is a minimal such graph, then the 
invariant forest either consists of a non-trivial orbit of $p$ edges 
or of a single fixed edge. 

In the first case, $\Gamma$ has $3p$ edges and $p+3$ vertices, 
and because the quotient $\Gamma/\Delta(E_2)$ collapses to 
$(\Theta_{p-1}\vee \Theta_{p-1})/\Delta(E_2)$,  
there are $3$ nontrivial orbits of edges and $4$ orbits 
of vertices, $3$ of which are trivial and one with $p$ vertices. 
Because all edges are moved, the valency of each fixed vertex must be at 
least $p$ and because each of the moving vertices has valency at least $3$, 
we see that, in fact, the valency of the fixed vertices is exactly $p$ 
and that of the others is $3$. It is then easy to check that 
admissibility implies that $\Gamma$ must be $K_{p,3}$.  

In the second case $\Gamma$ would $2p+1$ edges in one trivial and 
two non-trivial orbits and $4$ trivial orbits of vertices. 
It is easy to see that such a $\Gamma$ cannot be admissible.   

Finally we claim that there cannnot be any admissible 
$\Delta(E_2)$-graph $\Gamma$ with a $\Delta(E_2)$-invariant forest 
which collapses to $K_{p,3}$. In fact, such a minimal graph would 
either have $4p$ edges and $2p+3$ vertices, or $3p+1$ vertices and 
$p+4$ vertices. Again $\Gamma/\Delta(E_2)$ would have to collapse to 
$K_{p,3}/\Delta(E_2)$. Therefore, in the first case 
we would have $4$ nontrivial orbits of edges, $2$ 
nontrivial orbits of vertices and three trivial orbits of vertices. 
The valency of each fixed vertex would be at least $p$ 
and that of the others at least $3$, so that the sum of the valencies would be 
at least $3p+6p=9p$ which is in contradiction to having only $4p$ edges. 
In the second case we would have $3$ nontrivial and one trivial orbit of 
edges and $1$ nontrivial orbit and $4$ trivial orbits of vertices. 
The valency of at least $3$ of the fixed vertices would have to be at 
least $p$ and then the total valency would be at least $3p+3+3p=6p+3$ 
which contradicts having only $3p+1$ edges.  \ \ \ $\square$
\end{proof}
    
\bigskip
Finally we will need the intersection of the normalizers of $\Delta(E_2)$ 
and of $E_2$ resp. of $E_{0,p-1}$ and of $E_2$.  This is now straightforward 
to deduce just from the structure of $N_{Out(F_{2(p-1)})}(E_2)$ 
given in Proposition \ref{E2-normalizer}.  With notation as in the introduction 
we get the following result. 

\begin{proposition}\label{intersections} There are isomorphisms 
\begin{eqnarray}
N_{0,p-1}\cap N_2 &\cong &  
N_{\Sigma_p}(\Z/p)\times N_{\Sigma_p}(\Z/p) \nonumber \\
N_{\Delta}\cap N_2 &\cong &  
(\Z/p\times \Z/p)\rtimes (Aut(\Z/p)\times \Z/2) \ . \nonumber 
\end{eqnarray}
\end{proposition}

\section{Evaluation of the normalizer spectral sequence for $p>3$}

By section 4 the poset $\cal{A}$ of elementary abelian 
$p$-subgroups of $Out(F_{2(p-1)})$, $p>3$, consists of the orbits of $E_R$, 
of $E_{0,p-1-s}$ for $0\leq s\leq\frac{p-1}{2}$, 
of $\Delta(E_2)$, of $E_2$ and the orbits of the two edges 
formed by $\Delta(E_2)$ and $E_2$ resp. by $E_{0,p-1}$ and $E_2$. 
Therefore the quotient of $\mathcal{A}$ by the action of 
$Out(F_{2(p-1)})$ consists of singletons corresponding to the orbits of 
$E_R$ resp. of $E_{s,p-1-s}$ for $1\leq s\leq \frac{p-1}{2}$, and 
of one component of dimension $1$. This latter component 
has three vertices formed by the orbits of $E_{0,p-1}$, of 
$\Delta(E_2)$ and of $E_2$, and two edges formed by the orbits of the edge 
between $E_{0,p-1}$ and $E_2$ resp. by the edge between 
$\Delta(E_2)$ and $E_2$.  
We will denote the preimage of these components in $\mathcal{A}$ 
by $\mathcal{A}_R$ resp. by  $\mathcal{A}_{s,p-1-s}$ for 
$1\leq s\leq \frac{p-1}{2}$ 
resp. by  $\mathcal{A}_2$.   
\medbreak 

The contributions of $\mathcal{A}_R$ and of $\mathcal{A}_{s,p-1-s}$ 
to $H^*_{Out(F_{2(p-1)}}(\mathcal{A};\FF_p)$ are simply given by the 
cohomology of the corresponding normalizers. The more interesting part 
is given by the component $\mathcal{A}_2$ which is described in the 
following result. Together with Proposition \ref{rose-normalizer} 
and Proposition \ref{theta-normalizers} this finishes the proof of 
Theorem \ref{the:4}.

\begin{proposition} 
a) There is a canonical isomorphism 
$$
H^*_{Out(F_{2(p-1)})}(\mathcal{A}_2;\FF_p)\cong 
H^*(N_{0,p-1}*_{N_{0,p-1}\cap N_2}N_2;\FF_p)\ .
$$
b) The restriction map
to the orbit of $E_2$ induces an epimorphism of rings
$$
H^*_{Out(F_{2(p-1)})}(\mathcal{A}_2;\FF_p)\to 
H^*(N_2;\FF_3)\cong 
H^*(N_{\Sigma}(\Z/p)\wr \Z/2;\FF_p)
$$
whose kernel is isomorphic to the ideal
$H^*(N_{\Sigma_p}(\Z/p);\FF_p)\otimes K_{p-1}$ where
$K_{p-1}$ is the kernel of the restriction map
$H^*(Aut(F_{p-1});\FF_p)\to H^*(N_{\Sigma_p}(\Z/p);\FF_p)$.
\end{proposition}

\begin{proof} We consider the isotropy spectral sequence 
for the map 
$$
EOut(F_{2(p-1)})\times_{Out(F_{2(p-1)})}\mathcal{A}_2\to 
\mathcal{A}_2/Out(F_{2(p-1)})
$$ 
By Proposition \ref{Delta(E_2)} and Proposition \ref{intersections} 
the edge between $\Delta(E_2)$ and $E_2$ gives the same contribution 
as the vertex $\Delta(E_2)$ (because $p>3$) and the corresponding face map 
induces an isomorphism so that this edge may be ignored. Then (a) follows 
from Lemma \ref{amalgam}.  
\smallbreak 

Next we claim that the restriction map from $H^*(N_{0,p-1};\FF_p)$
to $H^*(N_{0,p-1}\cap N_2;\FF_p)$ is onto; 
in fact by Proposition 
\ref{theta-normalizers} and Proposition \ref{intersections} it is enough 
to show that the restriction map 
$$
H^*(Aut(F_{p-1});\FF_p)\to H^*(N_{\Sigma}(\Z/p);\FF_p)
$$ 
is onto.  This in turn follows from \cite{GMV} where it is shown 
that the restriction map is an isomorphism in Farrell cohomology 
together with the observation that the virtual cohomological dimension is 
$2p-5$ and the cohomology of $N_{\Sigma_p}(\Z/p)$ is trivial below 
dimension $2p-3$. 

From the spectral sequence we see now that 
$H^*_{Out(F_{2(p-1)}}(\mathcal{A}_2;\FF_p)$ is the equalizer 
of the two maps  
$$
H^*(N_{0,p-1};\FF_3)\times H^*(N_2;\FF_3)
\to H^*(N_{0,p-1}\cap N_2;\FF_3)
$$ 
and the result follows by using once more Proposition \ref{theta-normalizers} 
and Proposition \ref{intersections}.   $\square$
\end{proof}

\begin{remark} For $p=3$ the situation is somewhat more 
complicated due to the 
symmetry of the graph $K_{3,3}$ and the resulting additional 
elementary abelian subgroups of $Out(F_4)$. However, with a bit of 
effort one can carry out the same analysis for $p=3$ and thus 
get a second proof of Theorem \ref{the:4}. We leave it to the 
interested reader to work out the details.  
\end{remark}

\end{document}